\documentclass[11pt,twoside]{article}

\def\titlename{\huge Fourier algebras of hypergroups and central algebras on compact (quantum) groups}

\title{\titlename}

\def\authname{Mahmood Alaghmandan  and Jason Crann}
\author{{\normalsize\sc \authname}}

\usepackage{amsmath}

\usepackage{yfonts}

\usepackage{xcolor}
\definecolor{blue1}{RGB}{32,78,170}
\definecolor{blue2}{RGB}{93,92,160}
\definecolor{blue3}{RGB}{40,51,202}
\definecolor{blue4}{RGB}{0,0,0}
\definecolor{purple1}{RGB}{128,0,128}

\usepackage[all]{xy}

\definecolor{El}{rgb}{.4,.9,1}
\usepackage[pdftex,
            linkcolor=blue1,
            citecolor=blue3,
            colorlinks=true]{hyperref}

\usepackage{titlesec}

\titleformat{\section}
  {\normalfont\fontsize{12}{15}\bfseries}{\thesection}{1em}{}

\usepackage[english]{babel}
\usepackage{blindtext}
\usepackage[scaled=.90]{helvet}
\usepackage{xcolor}
\usepackage{titlesec}
\titleformat{\chapter}[display]
  {\normalfont\sffamily\huge\bfseries\color{blue4}}
  {\chaptertitlename\ \thechapter}{20pt}{\Huge}
\titleformat{\section}
  {\normalfont\sffamily\Large\color{blue4}}
{\thesection}{1em}{}

 \titleformat{\subsection}
  {\normalfont\sffamily\large\color{blue4}}
  {\thesubsection}{1em}{}

\newcommand{\ma}[1]{\emph{#1}}

            \newcounter{pulse}[section]

\numberwithin{pulse}{section}
\numberwithin{equation}{section}

\newtheorem{theorem}[pulse]{\bf \textsf{Theorem}}
\newtheorem{proposition}[pulse]{\bf \textsf{Proposition}}
\newtheorem{lemma}[pulse]{\bf \textsf{Lemma}}

\newtheorem{thm}[pulse]{\bf \textsf{Theorem}}
\newtheorem{prop}[pulse]{\bf \textsf{Proposition}}
\newtheorem{lem}[pulse]{\bf \textsf{Lemma}}
\newtheorem{cor}[pulse]{\bf \textsf{Corollary}}

\newtheorem{dummy-eg}[pulse]{\bf \textsf{Example}}
\newtheorem{dummy-rem}[pulse]{\bf \textsf{Remark}}

\newenvironment{rem}{\begin{dummy-rem}\upshape}{\end{dummy-rem}\ignorespacesafterend}

\newtheorem{dummy-def}[pulse]{\bf \textsf{Definition}}

\usepackage{amsmath,amssymb}

\newenvironment{dfn}{\begin{dummy-def}\upshape}{\end{dummy-def}\ignorespacesafterend}

\newenvironment{proof}{\noindent{\it Proof.}\/}{\hfill$\Box$\newline\ignorespacesafterend}

\newcommand{\supp}{\operatorname{supp}}

\newcommand{\conj}{{\operatorname{Conj}}}

\newcommand{\cB}{{\cal B}}

\newcommand{\om}{\omega} 

\newcommand{\norm}[1]{\Vert #1 \Vert}

\newcommand{\bm}{\operatorname{\bf m}}

\newcommand{\cH}{{\mathcal H}}

\newcommand{\VN}{\operatorname{VN}}
\newcommand{\ignore}[1]{}

\newcommand{\LO}{L^1(G)}
\newcommand{\LOQ}{L^1(\mathbb{G})}

\newcommand{\LTQ}{L^2(\mathbb{G})}

\newcommand{\LIQ}{L^{\infty}(\mathbb{G})}

\newcommand{\al}{\alpha}
\newcommand{\be}{\beta}
\newcommand{\vphi}{\varphi}
\newcommand{\Lphi}{\Lambda_\varphi}

\newcommand{\lm}{\lambda}

\newcommand{\ten}{\otimes}

\newcommand{\G}{\Bbb{G}}

\newcommand{\mc}[1]{\mathcal{#1}}

\newcommand{\la}{\langle}
\newcommand{\ra}{\rangle}

\newcommand{\Irr}{{\operatorname{Irr}(\mathbb{G})}}
\newcommand{\IrrG}{{\operatorname{Irr}( {G})}}
\newcommand{\wG}{\IrrG}

\newcommand{\M}{\Bbb{M}}

\begin{document}

\maketitle

\begin{abstract}
This paper concerns the study of regular Fourier hypergroups through multipliers of their associated Fourier algebras. We establish hypergroup analogues of well-known characterizations of group amenability, introduce a notion of weak amenability for hypergroups, and show that every discrete commutative hypergroup is weakly amenable with constant 1. Using similar techniques, we provide a sufficient condition for amenability of hypergroup Fourier algebras, which, as an immediate application, answers one direction of a conjecture of Azimifard--Samei--Spronk \cite{AzSaSp} on the amenability of $Z\LO$ for compact groups $G$. In the final section we consider Fourier algebras of hypergroups arising from compact quantum groups $\G$, and in particular, establish a completely isometric isomorphism with the center of the quantum group algebra for compact $\G$ of Kac type.
\end{abstract}

\section{Introduction}\label{s:introduction}

The Fourier algebra $A(G)$ of a locally compact group $G$ is a central object in abstract harmonic analysis, providing remarkable manifestations of analytical and topological properties of $G$. A well-known example is given by Leptin's theorem \cite{lep}, which says that a locally compact group $G$ is amenable if and only if its Fourier algebra $A(G)$ has a bounded approximate identity.

As in the setting of locally compact groups, the \ma{Fourier space} $A(H)$ of a hypergroup $H$ plays an important role in the harmonic analysis.
A hypergroup $H$ whose Fourier space forms a Banach algebra under pointwise multiplication is said to be a \ma{regular Fourier hypergroup} \cite{mu1}, in which case we refer to $A(H)$ as the \ma{Fourier algebra} of the hypergroup. In this work, we study regular Fourier hypergroups through multipliers of $A(H)$, denoted $MA(H)$. Building on work of the first author \cite{ma5}, who established a hypergroup analogue of Leptin's theorem for discrete groups, we show that a discrete hypergroup $H$ satisfies $(P_2)$ if and only if $MA(H)= B_\lambda(H)$ isometrically, where $B_\lambda(H)$ is the reduced Fourier--Stieltjes space of $H$. In the case where $H$ is a discrete group $G$, our techniques yield a new proof of the well-known equivalence of amenability of $G$ and the equality $MA(G)=B_\lambda(G)$. We also introduce a notion of weak amenability for hypergroups and show that every discrete commutative hypergroup is weakly amenable with constant 1. Recall that commutative hypergroups need not necessarily satisfy $(P_2)$, contrary to the group setting.

We then use multipliers to investigate amenability of $A(H)$ for discrete commutative hypergroups. Our main result, Theorem~\ref{t:A(H)-amenable}, gives a sufficient condition for amenability of $A(H)$ in terms of boundedness of the Haar measure. Through the isometric isomorphism $Z\LO\cong A(\mathrm{Irr}(G))$ between the center of a compact group algebra and the Fourier algebra of the associated hypergroup of irreducible representations, our main result shows that if a compact group has an open abelian subgroup then the center of the group algebra $Z\LO$ is amenable; thereby establishing one side of the main conjecture of \cite{AzSaSp}.

Finally, we prove that the center of the group algebra of a compact quantum group of Kac type is completely isometrically isomorphic to the Fourier algebra of the associated discrete hypergroup; a result which initiates many questions regarding the applications of discrete hypergroups and their Fourier algebras to compact quantum groups.

The paper is organized as follows. We start in Section~\ref{s:preliminaries} with preliminaries and a brief overview of the Fourier algebra of hypergroups.
 Section~\ref{s:MA&M-of-A(H)} begins with some observations on completely bounded multipliers of Fourier algebras, leading to  Subsection~\ref{ss:multipliers-A(H)}, where we focus on their relationship with the reduced Fourier--Stieltjes space for discrete and commutative hypergroups.
 In Subsection~\ref{ss:ai-A(H)} we introduce weak amenability for hypergroups and prove that every commutative hypergroup is weakly amenable with constant 1. We then study amenability of Fourier algebras of discrete commutative hypergroups in Section~\ref{s:AM-A(H)}. Applications to central algebras on compact and discrete groups are given respectively in Subsections~\ref{ss:compact-groups} and \ref{ss:discrete-groups}. Finally, in Section~\ref{s:CQG} we establish the aforementioned relationship with central subalgebras of compact quantum groups.

\section{Preliminaries and notations}\label{s:preliminaries}

\subsection{Hypergroups}

For a locally compact Hausdorff space $H$ we denote by $C_b(H)$ the space of bounded continuous complex valued functions on $H$, and by $C_c(H)$ and $C_0(H)$ the subspaces of functions with compact support and vanishing at infinity, respectively. We denote by $M(H)$ the space of bounded Radon measures on $H$. For any $\mu\in M(H)$, let $\supp(\mu)$ denote the support of $\mu$. For each $x\in H$, we also denote by $x$ the corresponding point mass in $M(H)$. Here we follow the definition of a hypergroup as given by R. I. Jewett \cite{je}. We refer the reader to \cite{bl} for details as well as known results on hypergroups.

\begin{dfn}\label{d:hypergroups}
A locally compact Hausdorff space $H$ is a \ma{hypergroup} if the following conditions are satisfied:
\begin{itemize}
\item[1.]{There exists a binary operation on $M(H)$ under which $M(H)$ is an algebra where $M(H) \times M(H) \ni (\mu,\nu) \mapsto \mu \cdot \nu \in M(H)$. Moreover, for every $x, y$ in $H$, $x \cdot y$ is a probability measure and the mapping $(x, y) \mapsto x \cdot y$ is continuous from $H\times  H$ into $M(H)$ equipped with the topology $\sigma(M(H), C_b(H))$.}
\item[2.]{There exists a (unique) element $e$ in $H$ such that $ e\cdot \mu  =\mu \cdot e = \mu$ for all $\mu \in M(H)$.}
\item[3.]{There exists a homeomorphism $x\mapsto \tilde{x}$ of $H$ onto $H$  satisfying $\tilde{\tilde{x}}=x$,  $(x\cdot y\tilde{)} = \tilde{y} \cdot \tilde{x}$, and $e \in \supp(x \cdot y)$ if and only if $x =\tilde{y}$ for all $x,y \in H$.}
\item[4.]{ The set $\supp( x\cdot y)$ is compact. Moreover, the mapping $( x , y )\mapsto \supp (x \cdot y)$
is continuous from $H\times H$ into   the space of all non-empty compact subsets
of $H$ equipped with the ``Michael topology".}
\end{itemize}
  \end{dfn}

For each pair $x,y \in H$ and $f \in C_0(H)$, we denote the integral $\int f d(x \cdot y)$ by $f(x \cdot y)$ and  the left translation $L_x f(y)$ is defined by $f(\tilde{x}\cdot y)$.
All hypergroups considered in this article are assumed to possess a left invariant Haar measure $\lambda$, and we use $\int_H dx$ to denote the integration with respect to $\lambda$. Note that the existence of Haar measures for commutative/compact/discrete hypergroups can be proved (see \cite{je}).

We let $(L^p(H), \norm{\cdot}_p)$, $1 \leq p \leq  \infty$, denote the usual $L^p$-space with respect to a fixed left Haar measure $\lambda$. For each $f \in L^1(H)$ and $g \in L^p(H)$ for $1 \leq p \leq \infty$, it follows that
 \[
 f \cdot_\lambda g (x):=\int_H  f( y)g(\tilde{y} \cdot x )dy
 \]
 belongs to $L^p(H)$ and $\norm{f \cdot_\lambda  g}_p \leq  \norm{f}_1 \norm{g}_p$.
 In particular,  $L^1(H)$ is a Banach $\ast$-algebra with this action and  the involution  $f^*(x)=\Delta(\tilde{x})f(\tilde{x})$ where $\Delta$ is the modular function with respect to the left Haar measure $\lambda$.

A hypergroup $H$ is \ma{commutative} if $\mu \cdot \nu= \nu \cdot \mu$ for every pair $\mu,\nu \in M(H)$.
We define $\widehat{H}:=\{ \chi \in C_b(H): \chi(x \cdot y) = \chi(x) \chi(y)\ \text{and}\ \chi(\tilde{x})= \overline{\chi(x)}\}$,
equipped with the topology of uniform convergence on compact subsets of  $H$. One can  define Fourier and Fourier-Stieltjes transforms on $L^1(H)$ and $M(H)$ respectively, similar to the group case. In particular, there is a positive measure $\varpi$ on $\widehat{H}$, called the \ma{Plancherel measure}, such that the Fourier transform can be extended to a isometric isomorphism from  $L^2(H)$ onto $L^2(\supp(\varpi), \varpi)$.  Unlike the group case, $\widehat{H}$ does not necessarily form a hypergroup and $\supp(\varpi)$ may be a proper subset of $\widehat{H}$.

\subsection{Fourier spaces of hypergroups}

We now give a brief overview of the Fourier and Fourier--Stieltjes spaces of a hypergroup. We refer the reader to \cite{mu1} for details.
Let $\Sigma_{H}$ denote the equivalence classes of all representations of $H$ (see \cite{bl}). Abusing notation, let $\lambda$ denote the left regular representation of $H$ on $L^2(H)$ given by
$\lambda(x)\xi(y)=\xi (\tilde{x}\cdot y)$ for all $x,y\in H$ and for all $\xi \in L^2(H)$. This can be extended to $L^1(H)$ where $\lambda(f)\xi:=f \cdot_\lambda \xi$ for $f\in L^1(H)$ and $\xi \in L^2(H)$.
Let $C^*_\lambda(H)$ denote the  completion of $\lambda(L^1(H))$ in $\mathcal{B}(L^2(H))$ which is called the \ma{reduced $C^*$-algebra} of $H$. The von Neumann algebra generated by $\{\lambda(x): x \in H\}$ is called the \ma{von Neumannn algebra} of $H$, and is denoted by $\VN(H)$.
For any $f$ in $L^1(H)$, we define
\[
\norm{f}_{C^*(H)}: = \sup \{\norm{\pi(f)}: \ \pi \in \Sigma_H\}.
\]
The completion of $L^1(H)$ with respect to $\norm{\cdot}_{C^*(H)}$ is called the \ma{full $C^*$-algebra} of $H$ and is denoted by $C^*(H)$.
The dual of $C^*(H)$ is called the \ma{Fourier Stieltjes space} and is denoted by $B(H)$.
Let $B_\lambda(H)$ denote the dual of the reduced $C^*$-algebra $C^*_\lambda(H)$. It is known that $B_\lambda(H)$ can be realized as a closed subspace of $B(H)$. In particular, for a commutative hypergroup $H$, $B(H)$ is the closed linear span of $\widehat{H}$ while,  the linear space  of all elements in $\supp(\varpi)$ is dense in $B_\lambda(H)$.

The closed subspace of $B_\lambda(H)$ spanned by $\{ \xi\cdot_\lambda \tilde{\xi}: \xi\in C_c(H)\}$ is called the \ma{Fourier space} of $H$ and is denoted by $A(H)$. The dual of $A(H)$ can be identified with $\VN(H)$.
For a unimodular hypergroup $H$, it follows that $A(H)=\{ \xi \cdot_\lambda \tilde{\eta}: \ \xi,\eta\in L^2(H)\}$ and $\norm{u}_{A(H)}= \inf\norm{\xi}_{2} \norm{\eta}_2$, where the infimum is taken over all $\xi, \eta \in L^2(H)$ for which $u = \xi \cdot_\lambda \tilde{\eta}$ (see \cite{mu1} for commutative hypergroups and \cite[Proposition~1.3]{ma5} for general unimodular hypergroups).
Contrary to the group case (\cite{ey}), $B(H)$, $B_\lambda(H)$, and $A(H)$ do not necessarily form Banach algebras with respect to pointwise multiplication. If $A(H)$ forms a Banach algebra with respect to pointwise multiplication, the hypergroup $H$ is said to be a \ma{regular Fourier hypergroup}.



A hypergroup $H$ is said to satisfy property $(P_p)$, for $1\leq p < \infty$, if there is a net $(\xi_\alpha)_\alpha$ in $L^p(H)$ such that  $\xi_\alpha \geq 0$ and $\norm{\xi_\alpha}_p=1$ for all $\alpha$, and for every compact set $E\subseteq  H$,
\[
\lim_\alpha \norm{x \cdot \xi_\alpha - \xi_\alpha }_p =0\ \ \ \ \ (x\in E).
\]
The Reiter conditions $(P_p)$ and  other notions of amenability for hypergroups were studied extensively in \cite{ma5, izu, sk-la,  sk, singh-mem, ben1, ben2}.  Recall from \cite{sk} that for a commutative hypergroup $H$, the constant character $1$ belongs to  $\supp(\varpi)$ if and only if $H$ satisfies $(P_2)$. Based on these results, we see that property $(P_2)$ is an appropriate hypergroup analogue of group amenability. Indeed, it is known that $(P_1)\neq(P_2)$ for general hypergroups, so the existence of a left invariant mean on $L^{\infty}(H)$ does not capture the essential features of hypergroup amenability. Another justification is provided by the following result, which we will make extensive use of throughout the paper.

\begin{thm}\label{t:Leptin-thm}\cite[Theorem~3.3]{ma5}\\
Let $H$ be a discrete regular Fourier hypergroup. Then  $A(H)$ has a bounded approximate identity if and only if $H$ satisfies $(P_2)$. Further, in this case, $A(H)$ has a contractive approximate identity consisting of finitely supported positive definite elements.
 \end{thm}

\section{Multipliers and completely bounded multipliers of $A(H)$}\label{s:MA&M-of-A(H)}

 Recall that a Banach algebra $A$ equipped with an operator space structure is called  a \emph{completely bounded Banach algebra}  if there exits a constant $C\geq 1$ so that
 \[
 \| [a_{i,j} b_{k,l} ] \|_{\M_{mn}(A)} \leq C \| [a_{i,j}]\|_{\M_n(A)} \|[b_{k,l}]\|_{\M_m(A)}
 \]
 for every $ [a_{i,j}] \in \M_n(A)$ and $[b_{k,l}] \in \M_m(A)$ and $n,m \in \Bbb{N}$.  We say that $A$ is \emph{completely contractive} if $C=1$. Denoting by $\widehat{\otimes}$ the operator space projective tensor product, it follows that $A$ is completely bounded if and only if the linearization of the product
extends to a completely bounded map $m_{A} : A\widehat{\otimes} A\to A$. In what follows we adopt the notation of \cite{ruan-effros}, to which we refer the reader for details on operator spaces and completely bounded Banach algebras.

For a commutative Banach algebra $A$, a bounded linear operator $T\in\mathcal{B}(A)$ is a \emph{multiplier} of $A$ if  $T(a)b=aT(b)$ for every pair $a,b\in A$.  If $A$ is a completely bounded Banach algebra,  every multiplier $T$ which also belongs to $\mathcal{CB}(A)$ is called a \emph{completely bounded multiplier}. We use $MA$ and $M_{cb}A$ to denote, respectively,  the spaces of   bounded and completely bounded multipliers of $A$. Note that $MA$ is a closed subalgebra of $\mathcal{B}(A)$ and $M_{cb}A$ is a closed subalgebra of $\mathcal{CB}(A)$, so they form Banach algebras with the inherited norms.
It is known that for a commutative Banach algebra $A$, every element in $MA$ (and subsequently in $M_{cb}A$) can be identified uniquely with a bounded continuous function on the maximal ideal space of $A$ (see e.g. \cite[Proposition~2.2.16]{kaniuth}). Also, if $A$ is completely bounded, it follows that $A$ injects continuously into $M_{cb}A$.  In particular, if $A$ has a bounded approximate identity, the norms $\norm{\cdot}_A$, $\norm{\cdot}_{MA}$, and $\norm{\cdot}_{M_{cb}A}$ are equivalent on $A$; therefore, $A$ is a closed ideal in both $MA$ and $M_{cb}A$.

As the pre-dual of the hypergroup von Neumann algebra $\VN(H)$, the Fourier space $A(H)$ of any hypergroup inherits a natural operator space structure.

To study the norm of completely bounded multipliers of the Fourier algebra, we need the following hypergroup modification of a well-known result by de Canni{\`e}re  and Haagerup \cite[Theorem~1.6]{haa}. Here we use finite groups to avoid some difficulties in Subsection~\ref{ss:multipliers-A(H)}. The proof presented here is similar to the group case so we skip the details.

\begin{proposition}\label{p:SU(2)-in-multipliers}
Let $H$ be a regular Fourier hypergroup. Then $u \in M_{cb}A(H)$ if and only if $u\times 1_{G}$ belongs to $MA(H \times  G)$ for every finite group $G$ and $\sup_G \norm{u \times 1_{G} }_{MA(H \times G)} <\infty$. Further,
\[
\sup_G \norm{u \times 1_{G} }_{MA(H \times G)} =\norm{u}_{M_{cb}A(H)}.
\]
\end{proposition}

\begin{proof}
For each finite group $G$, $H\times G$ forms a hypergroup. Further, $ \VN(H \times G) $ coincides with $\VN(H)  \overline{\otimes} \VN(G)$.

Let $u \in M_{cb}A(H)$. The adjoint of $m_u(v) = uv$, $v\in A(H)$, denoted $M_u$, defines a completely bounded map from $VN(H)$ into $\VN(H)$. Therefore, by \cite[Lemma~1.5]{haa}, there exists a $\sigma$-weakly continuous operator $\tilde{M}_u$ on $\VN(H \times G) = \VN(H)  \bar{\otimes}  \VN(G)$  such that
\[
\tilde{M}_u(f \otimes g)= M_u(f) \otimes g\ \ (\text{for all } \ \ f\in \VN(H)\ \text{and}\ g\in \VN(G)).
\]
Moreover,
\[
\norm{u \otimes 1_G}_{MA(H\times G)} = \norm{\tilde{M}_u} \leq \norm{M_u}_{cb} = \norm{u}_{M_{cb}A(H)}.
\]

Conversely, for each $n\in \Bbb{N}$ there exists a finite group $G$ with an $n$-dimensional irreducible representation. Then $\VN(G)=\bigoplus_{i=1}^{m}  \M_{k_i} (\Bbb{C})$ where $k_i$'s are the dimension of irreducible representations of $G$, with $k_m=n$. Hence,
\[
\VN(H \times G) \cong \bigoplus_{i=1}^m  \M_{k_i}(\VN(H)).
\]
By restriction to the $m$th component of $\VN(H\times G)$, we have $\norm{M_u \otimes id_{n}} \leq \norm{M_{u \otimes 1_G} }$.
Thus,
\[
\norm{u}_{M_{cb}A(H)} = \norm{M_u}_{cb}\leq \sup_G \norm{M_{u \otimes 1_G}} =\sup_G \norm{u \otimes 1_G}_{MA(H \times G)}.
\]
\end{proof}

\begin{dfn}\label{d:completely-Fourier}
A regular Fourier hypergroup $H$ is called an \ma{operator Fourier hypergroup} if $A(H)$ forms a completely bounded Banach algebra under its canonical operator space structure.
\end{dfn}

It is known that every locally compact group is an operator Fourier hypergroup (see \cite[Section~16.2]{ruan-effros}). However, it is not clear that every regular Fourier hypergroup forms a completely bounded Banach algebra. In what follows we show that this is indeed the case for commutative hypergroups and discrete hypergroups arising from compact quantum groups of Kac type.

For two non-commutative regular Fourier hypergroups $H_1$ and $H_2$, it is not know whether the product hypergroup $H_1 \times H_2$ remains a regular Fourier hypergroup. In the following proposition we prove that the class of operator Fourier hypergroups is closed under products.

\begin{prop}\label{p:product}
Let $H_1$ and $H_2$ be two operator Fourier hypergroups. Then $H_1 \times H_2$ is also an operator Fourier hypergroup.
\end{prop}

\begin{proof}
By \cite[Theorem~7.2.4]{ruan-effros} we have the completely isometric identifications
$$A(H_1 \times H_2)\cong\VN(H_1 \times H_2)_*\cong(\VN(H_1) \bar{\otimes} \VN(H_2))_*\cong A(H_1) \widehat{\otimes} A(H_2).$$
By commutativity of $\widehat{\otimes}$ \cite[Theorem~7.1.4]{ruan-effros} we also have
\[
(A(H_1)\widehat{\otimes} A(H_2)) \widehat{\otimes} (A(H_1) \widehat{\otimes} A(H_2)) \cong (A(H_1)\widehat{\otimes} A(H_1)) \widehat{\otimes} (A(H_2) \widehat{\otimes} A(H_2)).
\]
Thus,
\[
m_1 \otimes  m_2 : (A(H_1)\widehat{\otimes} A(H_2)) \widehat{\otimes} (A(H_1) \widehat{\otimes} A(H_2)) \rightarrow A(H_1)\widehat{\otimes} A(H_2)
\]
is completely bounded (\cite[Corollary~7.1.3]{ruan-effros}). Therefore,  $A(H_1 \times H_2)$ is a completely bounded Banach algebra.
\end{proof}

 \subsection{$MA(H)$ for commutative and discrete hypergroups}\label{ss:multipliers-A(H)}

For a locally compact group $G$, it is well-known that $MA(G)=B_{\lambda}(G)$ if and only if $G$ is amenable. In this section we prove a corresponding result for every commutative and/or discrete regular Fourier hypergroup $H$. Recall that there are commutative hypergroups which do not satisfy $(P_2)$ \cite{sk}.

\begin{proposition}\label{p:M(A(H))=B-lambda-commutative}
Let $H$ be a commutative regular Fourier hypergroup. Then $H$ is an operator Fourier hypergroup and $M_{cb}A(H) = MA(H)$ with equal norms. Also,  we get  $MA(H)=B_\lambda(H)$ with equal norms  if and only if $H$ satisfies $(P_2)$.
\end{proposition}

\begin{proof}
{As the pre-dual of a commutative von Neumann algebra, $A(H)$ is canonically equipped with its MAX operator space structure, and therefore the projective and operator space projective tensor products coincide (see \cite[p 316]{ruan-effros}). Moreover,
each multiplier $m_u: A(H)\rightarrow A(H)$ is completely bounded and $\norm{u}_{MA(H)}=\norm{u}_{M_{cb}A(H)}$ (see \cite[Section~3.3]{ruan-effros}).}

Since $H$ is commutative, $B_\lambda(H)$ is a Banach algebra and it follows that $MB_\lambda(H)=MA(H)$ \cite[Corollary~4.3]{mu1}.
As $\supp(\varpi)\subseteq B_\lambda(H)$, we see that $1\in B_\lambda(H)$ if and only if $H$ satisfies $(P_2)$. Therefore, $B_\lambda(H)=MA(H)=MB_\lambda(H)$ if and only if $H$ satisfies $(P_2)$.
\end{proof}

Recall that for a discrete regular Fourier hypergroup $H$, $A(H)$ is a regular semisimple Banach algebra whose maximal ideal space is $H$ \cite[Theorem~5.13]{mu1}.

 \begin{thm}\label{t:multiplier-A(H)}
 Let $H$ be a  discrete regular  Fourier hypergroup. Then $MA(H)= B_\lambda(H)$ with equal norms if and only if $H$ satisfies $(P_2)$. Further, if $H$ is a discrete operator Fourier hypergroup satisfying $(P_2)$, then $M_{cb}A(H)=B_\lambda(H)=MA(H)$ with equal norms.
\end{thm}

\begin{proof}
 If $MA(H)=B_\lambda(H)$, the constant function $1$, which we denoted by $\chi_0$, belongs to $B_\lambda(H)=C^*_\lambda(H)^*$.
Hence, $\chi_0$  defines a multiplicative functional on $L^1(H)$.
By density of $L^1(H)$ in $C^*_\lambda(H)$, for each $g\in C^*_\lambda(H)$, $\langle \chi_0, {\tilde{g}}\cdot_\lambda g\rangle = \langle \chi_0,  \tilde{g} \rangle \langle \chi_0, g\rangle = |\langle \chi_0, g \rangle |^2 \geq 0$.  Hence,  $\chi_0|_{C^*_\lambda(H)}$ is a positive (multiplicative) functional on $C^*_\lambda(H)$ with  $\norm{\chi_0}_{B_\lambda(H)}=1$.

Let $E$ be a state extension of $\chi_0$ to $\VN(H)$. Since the normal states of $\VN(H)$ are weak$^*$ dense in the set of all states of $\VN(H)$, we may find a net of states $(u_\beta)_\beta$ in $A(H)$ such that $u_\beta\rightarrow E$ in  the weak$^*$ topology.

By \cite[Lemma~3.2]{ma5}, $A(H)$ is an ideal in its second dual, so for each $u \in A(H)$, $uE\in A(H)$. Let $T \in \VN(H)$ be arbitrary and $(f_\alpha)_\alpha$ be a net in $L^1(H)$ such that $f_\alpha \rightarrow T$ in the weak$^*$ topology of $\VN(H)$. Then
\[
\lim_\beta \langle uu_\beta, T\rangle = \langle uE, T\rangle =   \lim_\alpha \langle \chi_0 , uf_\alpha \rangle = \lim_\alpha \langle u, f_\alpha\rangle = \langle u, T\rangle.
\]
Therefore, $uu_\beta \rightarrow u$ with respect to the weak topology $\sigma(A(H), \VN(H))$. By the standard convexity argument we may render a  net $(\varphi_\alpha)_\alpha \subseteq \overline{\operatorname{conv}}\{u_\beta\}$
that is a bounded  approximate identity of $A(H)$. Theorem~\ref{t:Leptin-thm} then yields the claim.

\vskip0.5em

Conversely, suppose that $H$ satisfies $(P_2)$.  $A(H)$ is an ideal in its second dual \cite[Lemma~3.2]{ma5} which possesses a bounded approximate identity by Theorem~\ref{t:Leptin-thm}. By \cite[Theorem~3.1]{ka-ul}, any such Banach algebra is a BSE-algebra, i.e., $u\in MA(H)$ if and only if
\[
\sup\left\{ \left| \sum_{i=1}^n \alpha_i u(x_i) \lambda(x_i)\right|:\; x_i\in H,\; \alpha_i\in \mathbb{C}, n\in \mathbb{N}\; \text{ such that }\left\| \sum_{i=1}^n \alpha_i  {x_i}\right\|_{A(H)^*}\leq 1 \right\}< \infty.
\]
To verify the last statement recall that $u(x) = \lambda(x)^{-1} \langle x , u\rangle$ for each $u\in A(H)$ and $x\in H$ where $\langle \cdot, \cdot\rangle$ denotes the dual product of $(A(H), \VN(H))$.
 But $\norm{f}_{A(H)^*}=\norm{f}_{C^*_\lambda(H)}$ for each $f\in C_c(H)$. Hence, $u\in B_\lambda(H)=C^*_\lambda(H)^*$ and subsequently by \cite[Lemma~1]{BSE}, $(B_\lambda(H), \norm{\cdot}_{B_\lambda(H)})$ is a semisimple Banach algebra.   Further, \cite[Corollary~6]{BSE} implies that $\norm{\cdot}_{MA(H)}$ and $\norm{\cdot}_{B_\lambda(H)}$ are equivalent, i.e. there are constants $c,d$ so that $c \norm{\cdot}_{MA(H)}\leq \norm{\cdot}_{B_\lambda(H)} \leq d  \norm{\cdot}_{MA(H)}$.
  Recall that by Theorem~\ref{t:Leptin-thm}, $A(H)$ has a contractive approximate identity. Hence, the argument in the proof of \cite[Corollary~5]{BSE} implies that $d =1$. To show that $c=1$, recall that $B_\lambda(H)$ is a Banach algebra; hence, $\norm{u}_{MB_\lambda(H)} \leq \norm{u}_{B_\lambda(H)}$ for every $u \in B_\lambda(H)$. Since $A(H)$ is a closed subalgebra of $B_\lambda(H)$, we get
\[
\norm{u}_{MA(H)}   \leq  \norm{u}_{MB_\lambda(H)} \leq \norm{u}_{B_\lambda(H)}
\]
for every $u \in B_\lambda(H)$. Therefore, $\norm{\cdot}_{MA(H)}= \norm{\cdot}_{B_\lambda(H)}$.

\vskip1.0em

Suppose that $H$ is an operator Fourier hypergroup satisfying $(P_2)$.
Note that by Proposition~\ref{p:product}, for every finite group $G$, the discrete hypergroup $H \times G$ is an operator Fourier hypergroup satisfying $(P_2)$.  So the first part of the proof holds for $H\times G$ and therefore $MA(H \times G) = B_\lambda(H \times G)$ with equal norms.

For each $u \in B_\lambda(H)$, note that $u\otimes 1_{G}$ is a function in $MA(H \times G)$, where $1_G$ is the constant function on $G$.
Then $\norm{u\times 1_{G}}_{MA(H \times G)} =  \norm{u\times 1_{G}}_{B_\lambda(H \times G)}= \norm{u}_{B_\lambda(H)}$. Thus, by  Proposition~\ref{p:SU(2)-in-multipliers} $ \norm{u}_{B_\lambda(H)} = \norm{u}_{M_{cb}A(H)}$. Therefore, $MA(H) = M_{cb}A(H)$ with equal norms.
\end{proof}

Here, we showed equality between $MA(H)$ and $B_\lambda(H)$ for commutative/discrete hypergroups satisfying $(P_2)$. A question of interest is  the existence of a similar relation for $B(H)$. In other words, it would be interesting to find a hypergroup satisfying $(P_2)$ while  $B_\lambda(H) \neq B(H)$.
In \cite{geb}, it is claimed (without a proof) that this is not the case for some classes of discrete hypergroups with controlled growing rates. We are also aware that such a hypergroup cannot be admitted by the fusion rules of  compact groups (see Subsection~\ref{ss:compact-groups})  or compact quantum group of Kac type (see Section~\ref{s:CQG}).

\subsection{Weak amenability of commutative hypergroups}\label{ss:ai-A(H)}

The following definition is motivated by the group setting \cite{haa3}.

\begin{dfn}\label{d:WA}
 An operator Fourier hypergroup $H$ is said to be \ma{weakly amenable with constant $C>0$} if there exists a net $(u_\alpha)_\alpha \subseteq  A(H) \cap C_c(H)$  such that $\norm{u_\alpha}_{M_{cb}A(H)} \leq  C$ and $uu_\alpha  \rightarrow u$ in $A(H)$.
 \end{dfn}

We now state the main result of this section.

\begin{thm}\label{t:commutative-WA}
 {  Let $H$ be a commutative  discrete  regular Fourier hypergroup. Then $H$ is weakly amenable with constant $1$.}
\end{thm}

To prove this theorem we use a technique due to Voit \cite{voit} on generating new hypergroups using positive characters. In \cite{voit} it was proved that for every commutative hypergroup $H$, there is a unique positive character $\chi_0$ in the support of the Plancherel measure satisfying $\sup_{\chi \in \supp(\varpi)}  |\chi(x)| \leq \chi_0(x)$
for every $ x \in H$.  One then defines a new hypergroup convolution on $H$, denoted here by $\circ$, via
\[
  x  \circ y := \frac{ \chi_0}{\chi_0( x \cdot y)} x \cdot y
\]
for every pair $x,y \in H$. We use $H_0$ to denote the later hypergroup structure on $H$ whose Haar measure is $\lambda'=\chi_0^2 \lambda$. In this case, it follows that $\widehat{H}_0$ is homeomorphic to the set $\{ \chi \in \widehat{H}: |\chi(x)| \leq \chi_0(x)\ \ \forall x\in H\}$ through the mapping $\chi \mapsto \chi/\chi_0$.  Note that  $H$ is discrete if and only $H_0$ is discrete.

\begin{lem}\label{l:Fourier-of-I0}
Let $H$ be a   commutative   hypergroup. Then $A(H)$ is isometrically Banach space isomorphic to $A(H_0)$ through the mapping $u \mapsto u/ \chi_0$.  Moreover, $u \mapsto u/\chi_0$ also yields an isometric Banach space isomorphism from $B_\lambda(H)$ onto $B_\lambda(H_0)$.
\end{lem}

\begin{proof}
\cite[Proposition~2.6]{voit} proves that  $L^2(H)$ is isometrically  isomorphic to $L^2(H_0)$ through $\xi \mapsto \xi/\chi_0$. Recall that $A(H)= L^2(H) \cdot_\lambda L^2(H)$ and $A(H_0)=L^2(H_0) \circ_{\lambda'}  L^2(H_0)$. For each pair $\xi  , \eta$ in $L^2(H_0)$,  note that
\begin{eqnarray*}
\xi \circ_{\lambda'}  \eta(z) &=& \int_H \xi(x) \eta(\tilde{x} \circ z) \chi_0(x)^2 dx\\
&=& \int_H \xi(x) \int_H \frac{\eta(y) \chi_0(y)}{\chi_0(\tilde{x}\cdot z)} d(\tilde{x}\cdot z)(y) \chi_0(x)^2  dx\\
&=& \int_H \xi(x) \frac{(\chi_0 \eta)(\tilde{x} \cdot z)}{\chi_0(\tilde{x}) \chi_0(z)} \chi_0(x)^2 dx\\
&=& \frac{1}{\chi_0(z)} \int_H (\chi_0 \xi)(x) (\chi_0 \eta)(\tilde{x}\cdot z) dx = \frac{(\chi_0 \xi) \cdot_\lambda (\chi_0 \eta) }{\chi_0}(z).
\end{eqnarray*}
Hence,  $u \mapsto u/\chi_0$ is  a  mapping from $A(H)=L^2(H) \cdot_\lambda L^2(H)$ onto $A(H_0)=L^2(H_0) \circ_{\lambda'}  L^2(H_0)$ satisfying
\begin{eqnarray*}
\norm{u/\chi_0}_{A(H_0)} &=& \inf\{ \norm{\xi }_{L^2(H_0)} \norm{\eta }_{L^2(H_0)} : \ \ \  {u}/{\chi_0}=\xi\circ_{\lambda'}  \eta, \xi,\eta\in L^2(H_0)\}\\
&=& \inf\{ \norm{\xi }_{L^2(H)} \norm{\eta }_{L^2(H)}: \ \ \ u=\xi\cdot_\lambda  \eta, \xi,\eta\in L^2(H)\}\\
&=& \norm{u}_{A(H)}.
\end{eqnarray*}

To prove that every $u \in B_\lambda(H)$ is mapped to $u/\chi_0 \in B_\lambda(H_0)$, we use the fact that $\norm{\phi}_{C^*_\lambda(H_0)} = \norm{\chi_0 \phi}_{C^*_\lambda(H)}$ (see \cite[Proposition~2.6]{voit}). Indeed,
\begin{eqnarray*}
\norm{u/\chi_0}_{B_\lambda(H_0)} &=& \sup\left\{ \left|\int_H \frac{u(x)}{\chi_0(x)} \phi(x) \chi_0^2(x) dx\right| : \phi \in C_c(H),\ \norm{\phi}_{C^*_\lambda(H_0)}\leq 1 \right\}\\
&=& \sup\left\{ \left| \int_H  {u(x)}  \phi(x)   dx\right|  : \phi \in C_c(H),\ \norm{\phi}_{C^*_\lambda(H)}\leq 1 \right\}\\
&=& \norm{u}_{B_\lambda(H)}.
\end{eqnarray*}
\end{proof}

\begin{rem}\label{r:B-I-0=MA(H)}
Let $H$ be a commutative regular Fourier hypergroup.
\begin{itemize}
\item[1.]{For each $u \in A(H)$, $\chi_0 u \in A(H)$ ($\chi_0 \in B_\lambda(H) \subseteq MA(H)$)  and therefore, $u = \chi_0 u /\chi_0 \in A(H_0)$ by Lemma~\ref{l:Fourier-of-I0}. Therefore, $A(H)$ is a subset of $A(H_0)$. Indeed,
\[
\norm{u}_{A(H_0)} = \norm{u\chi_0 }_{A(H)} \leq \norm{\chi_0}_{B_\lambda(H)} \norm{u}_{A(H)} \leq \norm{u}_{A(H)}
\]
 for every $u \in A(H)$. }
 \item[2.]{$C_c(H) \cap A(H) =  C_c(H) \cap A(H_0)$ and it is  a dense subset in both of $A(H)$ and $A(H_0)$. The inclusion $C_c(H) \cap A(H) \subseteq C_c(H) \cap A(H_0)$ follows by the first part of the remark. To prove the converse, suppose that $u \in C_c(H) \cap A(H_0)$. Based on  Lemma~\ref{l:Fourier-of-I0}, there is some $v \in C_c(H) \cap A(H)$ so that $u= v/\chi_0$. But since $|\chi_0(x)|\geq \delta>0$ for some $\delta>0$ and every $x\in \supp(v)$, by Theorem~3.6.15 and Theorem~3.7.1 of \cite{ric}, there is some $\phi \in MA(H)$ so that $\phi(x)=\chi_0(x)^{-1}$ for each $x\in \supp(v)$. Hence, $u = v/\chi_0 = v \phi \in A(H)$.}
 \item[3.]{For each $\phi \in MA(H)$, $\phi$ also belongs to $MA(H_0)$. To prove this claim note that
 \[
 \norm{\phi u}_{A(H_0)} = \norm{\phi \chi_0 u}_{A(H)} \leq \norm{\phi}_{MA(H)} \norm{\chi_0 u}_{A(H)} = \norm{\phi}_{MA(H)} \norm{u}_{A(H_0)}
 \]
 for every $\phi \in MA(H)$ and $u \in C_c(H) \cap A(H)$. Conversely, for $\phi \in MA(H_0)$ and $u \in C_c(H) \cap A(H)$, we get
 \[
 \norm{\phi u}_{A(H)} = \norm{\phi \frac{u}{\chi_0}}_{A(H_0)} \leq \norm{\phi}_{MA(H_0)} \norm{\frac{u}{\chi_0}}_{A(H_0)}= \norm{\phi}_{MA(H_0)} \norm{ {u} }_{A(H)}.
 \] Therefore, two approximation arguments imply that $MA(H)  = MA(H_0) $ and moreover for every  $\phi \in MA(H_0) $ we have $\norm{\phi}_{MA(H_0)} =  \norm{\phi}_{MA(H)}$.}
 \item[4.]{By \cite[Remark~3.5]{mu1}, we know that every multiplier of $A(H_0)$ is a multiplier of $B_\lambda(H_0)$. But the constant function $1 \in B_\lambda(H_0)$; thus,
 \[
 B_\lambda(H) \subset MA(H) = MA(H_0) \subseteq B_\lambda(H_0).
 \]
 Similar to the argument in the first part of the remark, we get
$ \norm{u}_{B_\lambda(H_0)} \leq  \norm{u}_{B_\lambda(H)}$
 for every $u \in B_\lambda(H)$.}
\end{itemize}
 \end{rem}

\begin{proposition}\label{p:I-0-regular}
Let $H$ be a  commutative regular Fourier  hypergroup. Then $H_0$ is also a regular Fourier hypergroup.
\end{proposition}

\begin{proof}
Let $\varpi'$ denote the Plancherel measure of $H_0$. To prove this theorem, we use a characterization of  regular  Fourier hypergroups developed in \cite{mu1} which states that the hypergroup $H_0$ is regular  Fourier if and only for every pair $\chi'_1,\chi'_2 \in \supp(\varpi')$, $\chi'_1\chi'_2 \in B_\lambda(H_0)$ and $\norm{\chi'_1\chi'_2}_{B_\lambda(H_0)}\leq 1$.
By Lemma~\ref{l:Fourier-of-I0} and \cite[Proposition~2.5]{voit}, it suffices to show that for each pair $\chi_1, \chi_2 \in \supp(\varpi)$, $\chi_1\chi_2 /\chi_0 \in B_\lambda(H)$ and $\norm{\chi_1\chi_2/\chi_0}_{B_\lambda(H)}\leq 1$.

Let $f\in C_c(H)$. Recall that for $\chi_1 \in \supp(\varpi)$, $\norm{\chi_1/\chi_0}_\infty \leq 1$ (see \cite{voit}). Hence,   we have
\begin{eqnarray*}
\norm{\frac{\chi_1}{\chi_0} f}_{C^*_\lambda(H)} = \norm{\frac{\chi_1}{\chi_0} \frac{f}{\chi_0}}_{C^*_\lambda(H_0)} \leq  \norm{\frac{\chi_1}{\chi_0}}_{B_\lambda(H_0)} \norm{\frac{f}{\chi_0}}_{C^*_\lambda(H_0)}  = \norm{\chi_1}_{B_\lambda(H)} \norm{f}_{C^*_\lambda(H)}.
\end{eqnarray*}
Therefore, $\chi_1/\chi_0$ can be extended to a continuous operator on $C^*_\lambda(H)$ with $\norm{\chi_1/\chi_0}_{B(C^*_\lambda(H))}\leq \norm{\chi_1}_{B_\lambda(H)}=1$. Using this bounded operator,  for $\chi_2 \in \supp(\varpi)$, we define $\chi_1 \chi_2 / \chi_0 \in B_\lambda(H)$ by
\[
\langle \frac{\chi_1 \chi_2}{\chi_0}, f\rangle   := \langle \chi_2, \frac{\chi_1}{\chi_0} f\rangle  \ \ \ \ \ \ \ (f\in C_\lambda^*(H)),
\]
where $\langle\cdot,\cdot\rangle$ denotes the duality product of $(B_\lambda(H),C^*_\lambda(H))$. Then
\[
\norm{\chi_1 \chi_2/\chi_0}_{B_\lambda(H)} \leq \norm{\chi_1/\chi_0}_{B(C^*_\lambda(H))} \norm{\chi_2}_{B_\lambda(H)} \leq 1.
\]
This finishes the proof.

\end{proof}

\begin{rem}\label{r:MA(H)=B(I0)}
In the proof of Proposition~\ref{p:I-0-regular}, we actually proved that $(MA(H), \norm{\cdot}_{MA(H)})$ is isometrically Banach algebra isomorphic to $(B_\lambda(H_0), \norm{\cdot}_{B_\lambda(H_0)})$. Indeed, $A(H_0)$ is the closure of $A(H)$ in its (completely bounded) multiplier norm.  This Banach algebra was formerly introduced and studied for locally compact groups in \cite{br2, br1, br3}.
\end{rem}

Note that in neither of the results we proved above we did not assume that $H$ is discrete.

\vskip1.5em

\noindent{\it Proof of Theorem~\ref{t:commutative-WA}.}
First note that by Proposition~\ref{p:I-0-regular}, $H_0$ is a  discrete regular Fourier hypergroup satisfying $(P_2)$ (as $1$ belongs to the support of the Plancherel measure). Therefore, by Theorem~\ref{t:Leptin-thm}, $A(H_0)$ has a contractive approximate identity   which lies in $c_c(H) \subseteq A(H)$.
For each $u \in A(H)$, we get
\[
\norm{ue_\alpha - u}_{A(H)} = \lim_{\alpha} \norm{\frac{u}{\chi_0} e_\alpha - \frac{u}{\chi_0}}_{A(H_0)} =0
\]
as $u/\chi_0 \in A(H_0)$.  Note that $\norm{\cdot}_{MA(H)} = \norm{\cdot}_{MA(H_0)}$, by Remark~\ref{r:MA(H)=B(I0)}, while, $\norm{\cdot}_{MA(H_0)} = \norm{\cdot}_{A(H_0)}$. Therefore, $\{e_\alpha\}_\alpha$ is norm bounded in  $\norm{\cdot}_{MA(H)}=\norm{\cdot}_{M_{cb}A(H)}$ by $1$ ( Proposition~\ref{p:M(A(H))=B-lambda-commutative}).
\hfill $\Box$

\vskip2.0em

\begin{rem}\label{r:not-Haagerup}
Note that in the proof of Theorem~\ref{t:commutative-WA}, the bounded approximate identity $(e_\alpha)_\alpha$ of $A(H_0)$ can be formed as a net of compactly  positive definite functions   $e_\alpha=\xi_\alpha \circ_{\lambda'} \tilde{\xi}_\alpha$ on $H_0$ for unit vectors $\xi_\alpha$ in $C_c(H) \subseteq L^2(H_0)$. However, the $e_\alpha = ((\chi_0 \xi_\alpha) \cdot_\lambda (\chi_0 \xi_\alpha\tilde{)})/\chi_0$ are not necessarily positive definite as functions on $H$. In other words, the proof of the theorem does not imply the existence a hypergroup version of Haagerup property for commutative hypergroups.
\end{rem}

 \begin{rem}\label{p:MA(H)-dual-Q(H)}
For a commutative hypergroup $H$, one can apply  a hypergroup adaptation of \cite[Proposition~1.10]{haa} to show that $M_{cb}A(H)=MA(H)$ is the dual of a Banach space. In this case the weak amenability of commutative hypergroups implies the existence of  a net $\{e_\alpha\}$ in $A(H)$ so that $e_\alpha \rightarrow 1$ in the weak$^*$  topology of $MA(H)$. This notion can be interpreted as a hypergroup analogue of the \ma{approximation property}.
  \end{rem}

\section{Amenability of the Fourier algebra of discrete hypergroups}\label{s:AM-A(H)}

In this section $H$ is a discrete  commutative regular Fourier hypergroup with normalized Haar measure $\lambda$ i.e. $\lambda(e)=1$.
   The following lemma  is a hypergroup analogue of \cite[Theorem~1.16]{AzSaSp}. Here, we use $1_E$ to denote the characteristic function of set a $E$.
    For the definition of \emph{operator amenability}, see \cite[Section~16.2]{ruan-effros}, for instance. Note that if $H$ is a commutative regular Fourier hypergroup, operator amenability of $A(H)$ corresponds to amenability, by an argument similar to the one in the proof of Proposition~\ref{p:M(A(H))=B-lambda-commutative}.

\begin{lem}\label{l:1diagonal-in-MA(HxH)}
Let $H$ be a discrete operator Fourier hypergroup and $\Delta = \{(x,x): x \in H\}$. Then the following conditions are equivalent.
\begin{itemize}
\item[$(i)$]{$A(H)$ is operator amenable.}
\item[$(ii)$]{$1_{\Delta}$ belongs to $MA(H \times H)$ and $H$ satisfies $(P_2)$.}
\end{itemize}
\end{lem}

\begin{proof}
Let ${\bf m}:A(H)\widehat{\otimes} A(H) \rightarrow A(H)$ be defined by ${\bf m}(u\otimes v)=uv$. Also define $u \cdot (v\otimes w):=(uv)\otimes w$ and $ (v\otimes w) \cdot u:=v\otimes (uw)$.
As is well-known, $A(H)$ is operator amenable if and only if $A(H\times H)=A(H)\widehat{\otimes} A(H)$ admits a bounded approximate diagonal, that is, a bounded net $(m_\alpha) \subseteq A(H\times H)$ satisfying $\norm{u\cdot m_\alpha - m_\alpha \cdot u}_{A(H\times H)} \rightarrow 0$ and $\norm{u {\bf m}(m_\alpha) - u}_{A(H)} \rightarrow 0$.

\vskip1.0em

$(i) \Rightarrow (ii)$. Since every operator amenable algebra has a bounded approximate identity, the condition $(P_2)$ is immediate from Theorem~\ref{t:Leptin-thm}, for both $H$ and $H\times H$. Hence, Theorem~\ref{t:multiplier-A(H)} implies that $MA(H\times H)=B_\lambda(H\times H)$. Let $(m_\alpha)_{\alpha}$ be a bounded approximate diagonal for $A(H)$. Then $(m_\alpha)_\alpha \subseteq MA(H\times H)$ is a bounded net. Also, for each pair $x,y \in H$, $x\neq y$, $m_\alpha(x,y) \rightarrow 0$ while $m_\alpha(x,x)\rightarrow 1$. Hence, as a bounded net in $MA(H\times H)$, $(m_\alpha)$ clusters to $1_{\Delta}$ in the weak$^*$ topology of $B_\lambda(H \times H)=C^*_\lambda(H \times H)^*$ (by Theorem~\ref{t:multiplier-A(H)}).

\vskip1.0em

$(ii) \Rightarrow (i)$. Since $H$ satisfies $(P_2)$, $A(H)$  has a bounded approximate identity $(e_\alpha)_{\alpha}$. Define,  $m_{\alpha}:= (e_\alpha \otimes e_\alpha) 1_\Delta$.
Clearly $(m_\alpha)$ is a bounded net in $A(H\times H)$. Also,
\[
{\bf m}(m_\alpha) = {\bf m}\left( \sum_{x  \in H} e_\alpha(x)^2 x \otimes x \right) = e_\alpha^2
\]
which is a bounded approximate identity of $A(H)$. Finally for each $u$ in $A(H)$
\begin{eqnarray*}
u \cdot m_\alpha - m_\alpha \cdot u &=& \sum_{x \in H}  u(x)  e_\alpha(x)^2 x \otimes x -
\sum_{x \in H} u(x)  e_\alpha(x)^2 x \otimes x =0.
\end{eqnarray*}
\end{proof}

In the sequel we shall need the following lemma which we present from \cite[Theorem~5.6]{ile-st}, modified to our case. Here, a \ma{dual Banach algebra} is in the sense of   \cite[Definition~1·1]{ru-dual}.

 \begin{lemma}\label{l:homomorphism-extension}
Let $A$ and $B$ be two Banach algebras and $\theta :A \rightarrow MB$  a bounded algebra homomorphism.   If $A$ has a contractive  approximate identity $(e_\alpha)_\alpha$ and $MB$ is a dual Banach algebra, then $\theta$ can be extended to a  bounded map $\theta : M A \rightarrow MB$ where for each $c\in MA$,
\begin{equation}\label{eq:extending-homomorphisms}
 \theta(c)  :=  w^*-\lim_\alpha \theta(ce_\alpha).
 \end{equation}
 \end{lemma}

\begin{rem}\label{r:ile-sto-considerations}
Let $H$ be a discrete regular  Fourier hypergroup satisfying $(P_2)$.
Then $MA(H)=B_\lambda(H)$ is the dual of $C^*_\lambda(H)$. Under the canonical pointwise action $C^*_\lambda(H)$ becomes a $B_\lambda(H)$-bimodule. It then follows that multiplication in $B_\lambda(H)$ is separately weak* continuous, making $B_\lambda(H)$ a dual Banach algebra.
\end{rem}

One of the  technique in the proof of the following theorem is inspired by \cite{polish}.

\begin{theorem}\label{t:A(H)-amenable}
Let $H$ be a discrete commutative hypergroup satisfying $(P_2)$, $B(H\times H) = B_\lambda(H\times H)$, and $|\{\lambda(x): x\in H\}|< \infty$. Then $A(H)$ is amenable.
\end{theorem}

\begin{proof}
By Lemma~\ref{l:1diagonal-in-MA(HxH)}, we need only show that $1_\Delta \in MA(H \times H)$. To this end, define $\pi: H\times H \rightarrow \mathcal{B}(L^2(H))$  by
\[
\pi(x,y) \xi =  \lambda(\tilde{x})\lambda(y)   \xi
\]
for each $\xi \in L^2(H)$ where $ \lambda(f)  \xi $ denotes the left regular representation of  $\ell^1(H)$ on $L^2(H)$ \cite[Theorem~6.2I]{je}. Since $H$ is commutative and $\ell^1(H \times H)=\ell^1(H) \otimes_\gamma \ell^1(H)$, it follows that $\pi$ is indeed a hypergroup representation of $H\times H$  \cite[Definition~2.1.1]{bl}.  Define $\psi(x,y) := \langle \pi(x,y)  e, e\rangle$ on $H \times H$.  Based on the definition of a  Haar measure on a discrete hypergroup (see \cite[Theorem~1.3.26]{bl}), we have
\[
\psi(x,y) = \left\{
\begin{array}{l l}
\frac{1}{\lambda(x)}& y=x\\
0& \text{otherwise}
\end{array}
\right.
\]
As a coefficient function of $\pi$, we have $\psi \in B(H \times H) =B_\lambda (H \times H) = MA(H\times H)$.

Let $\bm: A(H\times H) \rightarrow A(H)$  be the multiplication.
  Since $H$ satisfies  $(P_2)$, one can easily show that $H\times H$ does as well. Therefore, $A(H \times H)$ has a contractive approximate identity. By Remark~\ref{r:ile-sto-considerations}, $MA(H)=B_\lambda(H)$ is a dual Banach algebra, so by Lemma~\ref{l:homomorphism-extension}, $\bm$ can be extended to a homomorphism $\bm: MA(H\times H) \rightarrow MA(H)$.
   Indeed, (\ref{eq:extending-homomorphisms}) implies that for each $\rho \in MA(H \times H)$, $
  \bm(\rho)= \sum_{x \in H} \rho(x,x) x\in MA(H)$.  In particular, for the aforementioned $\psi \in B(H \times H)$,
  \[
  \bm(\psi) = \sum_{x\in H} \frac{1}{\lambda(x)} x \in MA(H).
  \]
Let $\phi:=\bm(\psi) \in MA(H)$.  Since $\{\lambda(x): x\in H\}$ is finite,  $A:=\{\phi(x): x\in H\}$ is a finite set excluding $0$.  Indeed, if $A = \{a_1,\ldots, a_k\}$ then  the polynomial $p(z) = (z -a_1)\dots(z - a_k)$ satisfies $p(\phi)(x) = 0$ for every $x \in H$. Therefore $p(\phi) = 0$,
which in turn yields that $p(\gamma(\phi)) = 0$ for every linear multiplicative functional $\gamma$ on $MA(H)$. Hence $\gamma(\phi)$ is a root of polynomial $p$ and therefore $\gamma(\phi)\neq 0$ for every linear multiplicative functional $\gamma$ of the unital commutative Banach algebra $MA(H)$.  But it is known that such a function is invertible, i.e. $\phi^{-1} \in MA(H)$, see \cite[p217]{katz}.
Now, it is easy to check  that $\phi^{-1} \otimes 1 \in MA(H\times H)$. Consequently,   $1_\Delta = (\phi^{-1} \otimes 1) \psi$ belongs to  $MA(H \times H)$,
which finishes the proof.
\end{proof}

We believe the converse of Theorem~\ref{t:A(H)-amenable} to be true, that is, for a discrete  commutative  regular Fourier hypergroup $H$ satisfying $(P_2)$ with unbounded Haar measure, $A(H)$ is not amenable. This conjecture is  based on the known examples in the group case. See \cite{ma5, AzSaSp} and our results in the following subsections.

\subsection{Application to compact groups}\label{ss:compact-groups}

In this subsection, we use Theorem~\ref{t:A(H)-amenable} to prove one side of Conjecture~0.1 from \cite{AzSaSp} which we present as the following corollary.

\begin{cor}\label{c:AM-ZL1(G)}
Let $G$ be a compact group with an open abelian subgroup. Then $ZL^1(G)$ is amenable.
\end{cor}

To prove this corollary let us start by some observations regarding hypergroup structures admitted by compact groups.   It is known that there are two hypergroup structures constructed  by a compact group $G$. The compact commutative hypergroup of conjugacy classes of $G$, denoted by $\conj(G)$, and the discrete  commutative hypergroup of classes of irreducible representations of $G$, denoted by $\IrrG$, with the Haar measure $\lambda(\pi)=d_\pi^2$ where $d_\pi$ is the dimension of $\pi \in \IrrG$. It is known that both of these hypergroup structures are regular Fourier and satisfy $(P_2)$ (see \cite{ma5,mu2}). Also, they are strong hypergroups and are the duals of each other (see \cite[Subsection~2.1]{ma14}). On the one hand, $A(\conj(G))$ is isometrically Banach algebra isomorphic to $ZA(G)$, the central Fourier algebra of $G$. For more on $ZA(G)$, see \cite{ma12}. On the other hand, $A(\IrrG)$ is isometrically Banach algebra isomorphic to $ZL^1(G)$, the center of the group algebra \cite[Theorem~3.7]{ma}.

It is known that $\operatorname{Irr}(G\times G)=\IrrG\times \IrrG$. To use Theorem~\ref{t:A(H)-amenable}, we  need to prove that $B(\operatorname{Irr}(G\times G)) =MA(\operatorname{Irr}(G\times G))$.

\begin{proposition}\label{p:B(^G)-multiplier-of-A(^G)}
For any compact group $G$, $B(\wG)$ is isometrically Banach algebra isomorphic to $ZM(G)$, the center of the measure algebra of $G$. Further, $A(\wG)$ is an ideal in $B(\wG)$ and $B(\wG)=MA(\wG)$.
\end{proposition}

This result is most likely known to experts, but since we could not find a reference, we present a proof.

\vskip0.5em

\begin{proof}
The mapping $T:L^1(\wG)\rightarrow ZC(G)$ given by
\begin{equation}
T(\pi) = d_\pi \chi_\pi
\end{equation}
is a *-homomorphism with dense range. It therefore extends to a surjective *-homomorphism $T:C^*(\wG)\rightarrow ZC(G)$. One can similarly construct a left inverse to $T$, so that $C^*(\wG)$ is isometrically isomorphic to $ZC(G)$. Thus, $B(\wG)=C^*(\wG)^*$ is isometrically isomorphic to $ZM(G)$ through the dual  mapping $T^*:ZM(G) \rightarrow B(\wG)$. We show that $T$ is multiplicative.

Since each $\mu \in ZM(G)$ commutes with every $f\in L^1(G)$, $\pi(\mu)=c_\pi^{(\mu)} I$, where $I$ is the identity matrix of $\cB(\cH_\pi)$ and
\[
c_\mu^{(\pi)} = \int_G \pi_{i,i}(x) d\mu(x) = \frac{1}{d_\pi} \int_G \chi_\pi(x) d\mu(x)
\]
for some $i\in 1, \ldots, d_\pi$.
Also, for a function $f\in L^\infty(\wG)$, $f(\pi)= d_\pi^{-2} \langle f, \pi\rangle$ where $\langle \cdot,\cdot\rangle$ denotes the dual action of $L^1(\wG)$ and $L^\infty(\wG)$.
Hence,
\begin{eqnarray*}
T^*(\mu * \nu)(\pi)  &=& \frac{1}{d_\pi^2} \langle \mu * \nu , d_\pi \chi_\pi\rangle = \frac{1}{d_\pi} \sum_{i=1}^{d_\pi}  [\pi(\mu * \nu)]_{i,i}\\
&=& \frac{1}{d_\pi} \sum_{i=1}^{d_\pi}  [\pi(\mu) \pi(\nu)]_{i,i} = \frac{1}{d_\pi} \sum_{i,k=1}^{d_\pi}  [\pi(\mu)]_{i,k} [\pi(\nu)]_{k,i}\\
&=& \frac{1}{d_\pi} \sum_{i=1}^{d_\pi} c_\pi^{(\mu)} c_\pi^{(\nu)} =  T^*(\mu)(\pi) T^*(\nu)(\pi).
\end{eqnarray*}
Using a bounded approximate identity in $ZL^1(G)$, the standard argument shows that $ZM(G)$ is the space of all multipliers of $ZL^1(G)$. Thus, $B(\wG)=M(A(\wG))$ and $A(\wG)$ is an ideal in $B(\wG)$.
\end{proof}

\vskip1.5em

\noindent{\it Proof of Corollary~\ref{c:AM-ZL1(G)}.}
Note that every  open abelian subgroup of a compact group  has only finitely many cosets. Therefore, $G$ is virtually abelian. By \cite{moo}, a (locally) compact group is virtually abelian if and only if there exists an upper bound on the dimensions of its irreducible (unitary) representations. Thus, $\{\lambda(\pi)=d_\pi^2: \pi \in \IrrG\}$ is a finite set, and the discrete hypergroup $\IrrG$ satisfies all the conditions of Theorem~\ref{t:A(H)-amenable}. Therefore, $A(\IrrG) \cong ZL^1(G)$ is amenable.
\hfill $\Box$

\vskip1.5em

\begin{rem}\label{r:error}
The main theorem of \cite{fam}, which is stated for compact hypergroups, claims to prove both sides of \cite[Conjecture~0.1]{AzSaSp}, and, in particular,  Corollary~\ref{c:AM-ZL1(G)}, using the theory of compact hypergroups. Unfortunately, there are gaps in each direction of the proof of that result. Therefore to our knowledge, the other side of the conjecture  remains open.

 Here we briefly mention the gaps in the argument presented in \cite{fam} using its notation and context. In the first computation line on p. 1617, we should have
\[
\langle \mu, f \otimes f\rangle = k_\pi^2\langle \widehat{\mu}, \widehat{f} \otimes \widehat{f}\rangle.
\]
The extra coefficient $k_\pi^2$ is missing, there. Hence the deduction that $k_\pi/d_\pi \leq \norm{\mu}$, cannot be made form these computations.

The strongest conclusion one gains from the computation which is intended to verify $(i)\Rightarrow (ii)$, on p. 1617, is that the sequence $(M_n)_n$ is bounded with respect to its duality with $L^2(K)\widehat{\otimes}L^2(K)$. Since this space is well-known to be the pre-dual of $B(L^2(K))$, this shows that the sequence $(M_n)_n$ is bounded as convolution operators. This is insufficient information to conclude it is a bounded sequence of measures.
\end{rem}

 \subsection{Application to discrete groups}\label{ss:discrete-groups}

Let $G$ be a discrete \ma{finite conjugacy} or \ma{FC} group, that is, for each conjugacy class $C$, $|C|<\infty$. Let $\conj(G)$ denote the set of all conjugacy classes of $G$. Define $ZA(G)$ to be the subalgebra of class functions in the Fourier algebra of $G$.

It is known that $\conj(G)$ forms a discrete commutative hypergroup \cite{bl}. For each $C\in \conj(G)$ we have $\lambda(C)=|C|$, where $\lambda$ is the normalized Haar measure of $\conj(G)$. It was shown in \cite{mu2} that $ZA(G)$ is isometrically isomorphic to the Fourier algebra of $\conj(G)$, and, therefore, $\conj(G)$ forms a regular Fourier hypergroup. As a strong hypergroup, $\conj(G)$ always satisfies $(P_2)$.
Moreover, Theorem~3.15 and Theorem~4.2 in \cite{mu2} imply that $ZB(G\times G)=B(\conj(G\times G))=MA(\conj(G\times G))$ as every FC group is amenable.

  A discrete  group $G$ is called a   \ma{finite  commutator group} or \ma{FD} if its derived subgroup is finite. It is immediate that for a  group $G$, and  every $C\in \conj(G)$,   $|C|\leq |G'|$ where $G'$ is the \ma{derived subgroup} of $G$. Therefore, the orders of conjugacy classes of an FD group are uniformly bounded by $|G'|$. The converse is also true, that is, for an FC group $G$, if the orders of conjugacy classes are uniformly bounded then $G$ is an FD group \cite[Theorem~14.5.11]{rob}. In other words, for every FD group $G$ the Haar measure on $\conj(G)$ is bounded. Theorem~\ref{t:A(H)-amenable} immediately yields the following corollary.

\begin{cor}\label{c:A(Conj(G))}
Let $G$ be a discrete FD group. Then $ZA(G)$ is amenable.
\end{cor}

 In subsequent work by the first author and Spronk, the space $ZA(G)$ is studied for larger classes of locally compact groups.

\begin{section}{Hypergroup structure of compact quantum groups}\label{s:CQG}

The intention of this section is to establish a connection between hypergroup theory and compact quantum groups. Although some of the results below are likely known to experts, we believe that our explicit presentation will help to bridge the two communities.

In what follows we adopt the notation of \cite{ac}. We also refer the reader to \cite{ac} for a relevant introduction to compact quantum groups, their irreducible co-representations and (quantum) characters.

Let $\G$ be a compact quantum group.  As in the case with compact groups, the irreducible characters of $\G$ play an important role in the harmonic analysis. For $\al\in\Irr$, we let $\chi^\al$ be the \emph{character} of $\al$, and we let $\chi^\al_q$ be the \emph{quantum character} of $\al$. The quantum characters (as well as the characters) satisfy the decomposition relations: $\chi_q^\al\chi_q^\be=\sum_{\gamma\in\Irr}N^{\gamma}_{\al\be}\chi^\gamma_q$, where $N^{\gamma}_{\al\be}$ is the multiplicity of $\gamma$ in the tensor product representation $\al\ten\be$.
It is easy to show that for every compact quantum group $\G$, the set of (equivalence classes of) irreducible co-representations, $\Irr$, admits a discrete hypergroup
via the multiplication
\[
\alpha \cdot \beta :=\sum_{\gamma \in \Irr} \frac{d_\gamma}{d_\alpha d_\beta} N_{\alpha,\beta}^\gamma \gamma
\]
where $d$ denotes the quantum dimension. We denote this hypergroup structure by  $(\Irr,d)$. A similar decomposition leads to a hypergroup structure using the regular dimension, $n$,  which is denoted by $(\Irr,n)$. It is known that $(\Irr,n)$ and $(\Irr,d)$ coincide if and only if $\G$ is of Kac type.

Let $\lm_n$ and $\lm_d$ denote the respective left regular representations. We also let $\mathrm{Char}_q(\G)=\overline{\mathrm{span}}\{\vphi^\al_q\mid\al\in\Irr\}$. Then $\mathrm{Char}_q(\G)$ is a closed ideal in $\mc{Z}(\LOQ)$. We believe, but have been unable to show in general, that $\mc{Z}(\LOQ)=\mathrm{Char}_q(\G)$.
In the following we appeal to the \ma{Frobenius Reciprocity} property of the hypergroups $(\Irr,n)$ and $(\Irr,d)$ which is equivalent to the {Frobenius Reciprocity} of the compact quantum group i.e. $N_{\alpha\beta}^\gamma= N_{\overline{\alpha},\gamma}^ \beta = N_{\gamma,\overline{\beta}}^\alpha$ for all $\alpha,\beta,\gamma \in \Irr$.

\begin{thm}\label{t:ZL1=A(I)} Let $\G$ be a compact quantum group of Kac type. Then $\mc{Z}(\LOQ)$ and $A(\Irr,n)$ are completely isometrically isomorphic as completely contractive Banach algebras. In particular, $(\Irr,n)$ is an operator Fourier hypergroup.
\end{thm}

\begin{proof} Recall from \cite{ac} that $Z\LIQ$ is the von Neumann subalgebra of $\LIQ$ generated by the characters. Clearly $Z\LIQ(Z\LTQ)\subseteq Z\LTQ$ and $Z\LIQ\rightarrow B(Z\LTQ)$ is a faithful representation. By definition of the hypergroup convolution we have
$$\lm_n(\delta_\al)\delta_\be=\sum_{\gamma\in\Irr}N_{\overline{\al}\gamma}^\be\frac{n_\al n_\be}{n_\gamma}\delta_\gamma.$$
The unitary $U_n:Z(\LTQ)\ni\Lphi(\chi^\al)\mapsto\frac{1}{n_\al}\delta_\al\in\ell^2(I,n^2)$
then satisfies
\begin{align*}U_n^*\lm_n(\delta_\al)U_n\Lphi(\chi^\be)&=\sum_{\gamma\in\Irr}N_{\overline{\al}\gamma}^\be\frac{n_\al}{n_\gamma}U^*_n(\delta_\gamma)\\
&=\sum_{\gamma\in\Irr}N_{\overline{\al}\gamma}^\be n_\al\Lphi(\chi^\gamma)\\
&=\sum_{\gamma\in\Irr}N_{\al\be}^\gamma n_\al\Lphi(\chi^\gamma)\\
&=n_\al\Lphi(\chi^\al\chi^\be)\\
&=n_\al\chi^\al\Lphi(\chi^\be)\\
\end{align*}
for all $\al,\be\in\Irr$. It follows that $U_n^*\lm_n(\delta_\al)U_n=n_\al\chi^\al$, and $Z\LIQ$ is spatially *-isomorphic to $\VN(\Irr,n)$. Thus, $Z(\LIQ)_*\cong A(\Irr,n)$, completely isometrically.

Since the restriction of the Haar state $\vphi|_{Z\LIQ}$ is a normal faithful trace, we have $Z\LIQ_*\cong \mc{Z}(\LOQ)$. Under this identification, the pre-adjoint of the spatial *-isomorphism satisfies $U_n^*\cdot\delta_\al\cdot U_n=n_\al\vphi^\al$, $\al\in\Irr$, as is easily verified by invoking the duality relation
$$\la\delta_\al,\lm_n(\delta_\be)\ra_{A(\Irr,n),VN(\Irr,n)}=\delta_{\al,\overline{\be}}n_\al^2, \ \ \ \al,\be\in\Irr.$$
Equation (2.4) in \cite{ac} then implies that $U_n^*(\cdot)U_n:A(\Irr,n)\rightarrow\mc{Z}(\LOQ)$ is multiplicative. Therefore, $(\Irr,n)$ is an operator Fourier hypergroup.
\end{proof}

\begin{rem} The above isomorphism $\mc{Z}(\LOQ)\cong A(\Irr,n)$ identifies $f\in \mc{Z}(\LOQ)$ with $\hat{f}\in A(\Irr,n)$ given by
\begin{equation}\label{e:hat}\hat{f}(\al)=\frac{1}{n_\al}\la f,\chi^{\overline{\al}}\ra, \ \ \ \al\in\Irr.\end{equation}
\end{rem}

For general compact quantum groups, there is an injective contraction from $A(\Irr,d)$ into the center $\mc{Z}(\LOQ)$. To show this, first observe by the orthogonality relations that for $f\in\mathrm{span}\{\vphi^\al_q\mid\al\in\Irr\}$

\begin{equation}\label{e:Zform1}f\star u_{ij}^\al=\frac{1}{d_\al}\la f,\chi^{\overline{\al}}_q\ra u_{ij}^\al\end{equation}
for all $\al\in\Irr$, $1\leq i,j\leq n_\al$. This yields the equality
\begin{equation}\label{e:Zform2} f=\sum_{\al\in\Irr}d_\al f\star\vphi_q^\al=\sum_{\al\in\Irr}\la f,\chi^{\overline{\al}}_q\ra\vphi_q^\al.\end{equation}

\begin{prop} Let $\G$ be a compact quantum group. Then there exists a contractive injection $T:A(\Irr,d)\rightarrow\mc{Z}(\LOQ)$.
\end{prop}

\begin{proof} First, let $\al,\be\in\Irr$, and consider the functional $\om_{\Lphi(\chi^\al_q),\Lphi(\chi^\be_q)}|_{\LIQ}$. For $x\in\mc{A}$, we have
\begin{align*}\om_{\Lphi(\chi^\al_q),\Lphi(\chi^\be_q)}(x)&=\vphi((\chi^\be_q)^*x\chi^\al_q)=\vphi(x\chi^\al_q\sigma_{-i}((\chi^\be_q)^*))\\
&=\vphi(x\chi^\al_q\sigma_{i}(\chi^\be_q)^*)=\vphi(x\chi^\al_q\chi^{\overline{\be}}_q)\\
&=\sum_{\gamma\in\Irr}N^\gamma_{\al\overline{\be}}\vphi(x\chi^\gamma_q)\\
&=\sum_{\gamma\in\Irr}N^\gamma_{\al\overline{\be}}\vphi^\gamma_q(x).\end{align*}
By weak* density of $\mc{A}$ in $\LIQ$, it follows that $\om_{\Lphi(\chi^\al_q),\Lphi(\chi^\be_q)}|_{\LIQ}\in\mathrm{Char}_q(\G)$. Thus, by linearity and density, we have the following inclusion
$$\{\om_{\xi,\eta}|_{\LIQ}\mid\xi,\eta\in\mc{Z}(\LTQ)\}\subseteq\mathrm{Char}_q(\G).$$
Since the left hand side clearly contains $\mathrm{span}\{\vphi^\al_q\mid\al\in\Irr\}$, we obtain
$$\mathrm{Char}_q(\G)=\overline{\{\om_{\xi,\eta}|_{\LIQ}\mid\xi,\eta\in\mc{Z}(\LTQ)\}}.$$

Clearly, $\chi^\al_q(\mc{Z}(\LTQ))\subseteq\mc{Z}(\LTQ)$ for every
$\al\in\Irr$. Moreover, for any $\be,\gamma\in\Irr$, we have
\begin{align*}\la\chi_q^{\al^*}\Lphi(\chi^\be_q),\Lphi(\chi^\gamma_q)\ra&=\la\Lphi(\chi^\be_q),\Lphi(\chi^\al_q\chi^\gamma_q)\ra\\
&=\sum_{\delta\in\Irr}N^{\delta}_{\al\gamma}\la\Lphi(\chi^\be_q)\Lphi(\chi^{\delta}_q)\ra\\
&= \ N^{\beta}_{\al\gamma}= \ N^{\gamma}_{\overline{\al}\beta}\\
&=\sum_{\gamma\in\Irr}N^{\delta}_{\overline{\al}\beta}\la\Lphi(\chi^\delta_q)\Lphi(\chi^{\gamma}_q)\ra\\
&=\la\chi^{\overline{\al}}_q\Lphi(\chi_q^\be),\Lphi(\chi^{\gamma}_q)\ra.
\end{align*}
Hence, in $\mc{B}(\mc{Z}(\LTQ))$, we have $\chi_q^{\al^*}=\chi_q^{\overline{\al}}$, so that $\mc{B}_q:=\text{span}\{\chi^\al_q\mid\al\in\Irr\}$ is a self-adjoint unital subalgebra of $\mc{B}(\mc{Z}(\LTQ))$. Let $\mc{Z}\LIQ$ denote the von Neumann algebra generated by $\mc{B}_q$.

By definition of the hypergroup convolution and Frobenius reciprocity, it follows that the unitary
$$U_d:\mc{Z}(\LTQ)\ni\Lphi(\chi^\al_q)\mapsto\frac{1}{d_\al}\delta_\al\in\ell^2(I,d^2)$$
intertwines $\lm_d(\delta_\al)$ and $d_\al\chi^\al_q$, and yields a spatial *-isomorphism $\mc{Z}\LIQ\cong VN(\Irr,d)$.

Given a finitely supported $u\in A(\Irr,d)$, define $T(u)=\sum_{\al\in\Irr}u(\al)d_\al\vphi_q^\al$. If $\xi,\eta\in\mc{Z}(\LTQ)$ such that $\om_{\xi,\eta}|_{\LIQ}=T(u)$, then
\begin{align*}\om_{U_d\xi,U_d\eta}(\lm(\delta_\al))&=d_\al\la\chi^\al_q\xi,\eta\ra\\
&=\sum_{\be\in\Irr}u(\be)d_\be\la\vphi_q^\be,\chi^\al_q\ra\\
&=\sum_{\be\in\Irr}u(\be)\la\delta_\beta,\lm(\delta_\al)\ra\\
&=\la u,\lm_d(\delta_\al)\ra.\end{align*}
Thus, $\om_{U_d\xi,U_d\eta}|_{VN(\Irr,d)}=u$. Conversely, if $\om_{U_d\xi,U_d\eta}|_{VN(\Irr,d)}=u$, then since
$$u(\be)=\frac{1}{d_\be^2}\la u,\lm_d(\delta_{\overline{\be}})\ra_{A(\Irr,d),VN(\Irr,d)}, \ \ \ \be\in\Irr,$$
equation (\ref{e:Zform2}) implies $\om_{\xi,\eta}|_{\LIQ}=T(u)$. Putting things together, and recalling that $\LOQ$ is standardly represented on $\LTQ$, we have
\begin{align*}\norm{T(u)}_{\LOQ}&\leq\inf\{\norm{\xi}\norm{\eta}\mid\xi,\eta\in\LTQ, \ \om_{\xi,\eta}|_{\LIQ}=T(u)\}\\
&\leq\inf\{\norm{\xi}\norm{\eta}\mid\xi,\eta\in\mc{Z}(\LTQ), \ \om_{\xi,\eta}|_{\LIQ}=T(u)\}\\
&=\inf\{\norm{U_d\xi}\norm{U_d\eta}\mid\xi,\eta\in\mc{Z}(\LTQ), \ \om_{U_d\xi,U_d\eta}=u\}\\
&=\norm{u}_{A(\Irr,d)}.
\end{align*}
Hence, $T$ extends to a linear contraction $T:A(\Irr,d)\rightarrow\mc{Z}(\LOQ)$ such that $T(\delta_\al)=d_\al\vphi^\al_q$. Thus, $T$ is multiplicative on finitely supported functions by equation (2.4) of \cite{ac}, and it follows that $T(u)=\om_{\xi,\eta}|_{\LIQ}$ whenever $u=\om_{U_d\xi,U_d\eta}|_{VN(\Irr,d)}$, $\xi,\eta\in\mc{Z}(\LTQ)$. If $T(u)=0$, and $u=\om_{U_d\xi,U_d\eta}|_{VN(\Irr,d)}$ for some $\xi,\eta\in\mc{Z}(\LTQ)$, then $\om_{\xi,\eta}(\chi^\al_q)=0$ for all $\al\in\Irr$, and $u=0$.
\end{proof}

For general compact $\G$, it is known that $\G$ is co-amenable (i.e., $\LOQ$ has a bounded approximate identity) if and only if $(\Irr,n)$ satisfies $(P_2)$ \cite[Theorem 2.1.7]{nesh}. If, in addition, $\G$ of Kac type, then weaker approximation properties of $\LOQ$, such as weak amenability or the approximation property, entail the corresponding approximation property of its center $\mc{Z}(\LOQ)$ \cite{ruan}, and therefore, of the hypergroup $(\Irr,n)$. A natural question is whether weaker approximation properties of the hypergroups $(\Irr,n)$ and $(\Irr,d)$ entail the corresponding properties for $\LOQ$. For instance, Theorem~\ref{t:commutative-WA} implies that $(\Irr,n)$ is weakly amenable for any compact Kac algebra $\G$ with commutative fusion rules. By Theorem~\ref{t:ZL1=A(I)}, $\mc{Z}(\LOQ)$ has a multiplier bounded approximate identity. Does $\LOQ$ then necessarily have a multiplier bounded approximate identity?

Another question of interest is to study the character $\chi_0$ (in Subsection~\ref{ss:ai-A(H)}) and the corresponding hypergroups $(\Irr_0,n)$ and $(\Irr_0,d)$ (which satisfy $(P_2)$) from Voit's construction for arbitrary compact $\G$ with commutative fusion rules.

\end{section}

 \section*{Acknowledgement}

The first named author would like to express his deep gratitude to Yemon Choi, Ebrahim Samei, and Nico Spronk for several constructive discussions on the topic of Section~\ref{s:AM-A(H)}. This project was initiated at the Fields Institute during the Thematic Program on Abstract Harmonic Analysis, Banach and Operator Algebras in 2014. Parts of this research were carried out while the first author was visiting University of Waterloo and Carleton University.
We are grateful  for these kind hospitalities.

\vskip1.0em

\footnotesize
\def\cprime{$'$} \def\cprime{$'$} \def\cprime{$'$}

\vskip1.5em

\noindent {\bf Addresses:}

\vskip0.5em

\noindent{\bf Mahmood Alaghmandan}

\noindent Department of Mathematical Sciences, Chalmers University of Technology and University of Gothenburg, Gothen-
burg SE-412 96, Sweden

\noindent   E-mail address:    mahala@chalmers.se

 \vskip0.5em

  \noindent {\bf  Jason Crann}

 \noindent Institute for Quantum Computing, University of Waterloo, Waterloo, ON, Canada N2L 3G1

  \noindent Department of Pure Mathematics, University of Waterloo, Waterloo, ON, Canada N2L 3G1

   \noindent Department of Mathematics and Statistics, University of Guelph, Guelph, ON, Canada N1G 2W1

\noindent      E-mail address: jcrann@uwaterloo.ca

\end{document}